\documentclass[review]{elsarticle}

\usepackage{amsmath}
\usepackage{graphicx,setspace}
\usepackage{amsfonts,amsmath, amssymb, amsthm}
\usepackage{mathrsfs}
\usepackage{epstopdf}
\usepackage{float}
\usepackage{color}
\usepackage{subfigure}
\usepackage{bm}
\usepackage{graphicx}
\usepackage{amssymb, amsthm}
\usepackage{mathrsfs}
\usepackage{epstopdf}
\usepackage{float}
\usepackage{color}
\usepackage{subfigure}
\usepackage{cases}
\usepackage{geometry}
\usepackage{appendix}

\numberwithin{equation}{section}

\journal{?}

\begin{document}

\newtheorem{definition}{Definition}
\newtheorem{lemma}{Lemma}[section]
\newtheorem{remark}{Remark}[section]
\newtheorem{theorem}{Theorem}[section]
\newtheorem{proposition}{Proposition}[section]
\newtheorem{assumption}{Assumption}
\newtheorem{example}{Example}
\newtheorem{corollary}{Corollary}
\def\e{\varepsilon}
\def\Rn{\mathbb{R}^{n}}
\def\Rm{\mathbb{R}^{m}}
\def\Rd{\mathbb{R}^{d}}
\def\E{\mathbb{E}}
\def\P{\mathbb{P}}
\def\th{\theta}
\def\cC{{\mathcal C}}
\def\ssup{\sup_{t\in[0,T]}}
\numberwithin{equation}{section}

\begin{frontmatter}

\title{{\bf Homogenization of Dissipative Hamiltonian Systems under L\'evy Fluctuations
}}
\author{\centerline{\bf Zibo Wang$^{a,}\footnote{ zibowang@hust.edu.cn}$,
Li Lv$^{a,}\footnote{ Corresponding author: lilyu@hust.edu.cn }$,
and
Jinqiao Duan$^{b,}\footnote{duan@iit.edu}$}
\centerline{${}^a$ School of Mathematics and Statistics \& Center for Mathematical Sciences,}
\centerline{Huazhong University of Science and Technology, Wuhan 430074,  China}
\centerline{${}^b$ Department of Applied Mathematics,} \centerline{Illinois Institute of Technology, Chicago, IL 60616, USA}}

\begin{abstract}
  This work is devoted to deriving small mass limiting equation for a class of Hamiltonian systems with multiplicative L\'evy noise. Derivation of the limiting equation depends on the structure of the stochastic Hamiltonian systems, in which a noise-induced drift term arises. We prove convergence to the limiting equation in probability under appropriate assumptions on smoothness and boundedness. Furthermore, we demonstrate convergence in moment under stronger assumptions. A L\'evy type Smoluchowski-Kramers approximation result is presented as an illustrative example.
\end{abstract}

\begin{keyword}
 Homogenization; Hamiltonian systems; non-Gaussian L\'evy noise; noise-induced drift; small mass limit; effective reduction

\end{keyword}

\end{frontmatter}


\section{Introduction}

The motion of a diffusing particle of mass $m$ can be modeled by a stochastic differential equation (SDE)
\begin{equation*}
dq_t=v_tdt, \ \ \ \ mdv_t=-\gamma v_tdt+\sigma dW_t,
\end{equation*}
where $\gamma$ is the dissipation coefficient, $\sigma$ is the diffusion coefficient and $W$ is a Wiener process. The small mass limit problem was studied by Smoluchowski \cite{smoluchowski1916drei} and Kramers \cite{kramers1940brownian} when the mass $m \to 0$. Following their pioneering work, this subject has been investigated by a number of authors. For example, Nelson \cite{nelson2020dynamical} derived the limiting equation when $\gamma$ and $\sigma$ are constants and a Fokker-Planck equation approach was provided by Doering \cite{doering1990modeling}. Convergence in probability for $\gamma$ constant and $\sigma$ position-dependent was shown by Freidlin \cite{freidlin2004some}. For the infinite dimensional case, the problem was studied by Cerrai-Freidlin \cite{cerrai2006smoluchowski}. These above problems can be illustrated in the framework of homogenization, for which a splendid relevant reference is given \cite{pavliotis2008multiscale}.

Recently, the phenomenon of presence of noise-induced drift term in the small mass limit problem attracted wide attentions. It arises when the dissipation and diffusion coefficients depend on the state variable. Then there will be an additional drift term which does not appear in the original system. This phenomenon was firstly discovered by Hanggi \cite{hanggi1982nonlinear} for systems satisfying the fluctuation-dissipation relation. Then Volpe et al. \cite{volpe2010influence} made an experimental observation for this phenomenon. Hottovy et al. \cite{hottovy2015smoluchowski} derived the limiting equation of SDEs with arbitrary state-dependent friction. Birrell et al. developed small mass limit theory on compact Riemannian manifolds \cite{birrell2017small} and for Hamiltonian systems \cite{birrell2018homogenization}. A generalized homogenization theorem for Langevin systems was proved in \cite{birrell2019homogenization}. Lim et al. \cite{lim2020homogenization} discussed generalized Langevin equation for non-Markovian anomalous diffusions. We point out that most existing works mentioned above are for Gaussian noise.

However, random fluctuations in nonlinear dynamical systems are often non-Gaussian \cite{duan2015introduction}. The particle undergoing L\'evy superdiffusion is performing motion with random jumps and step lengths following a power-law distribution \cite{applebaum2009levy}. As an important kind of non-Gaussian noise, L\'evy noise have been found widely in atmospheric turbulence \cite{sanchez2008nature}, epidemic spreading \cite{dybiec2009modelling} and cell biological behaviour \cite{xu2016switch}. L\'evy noise-driven non-equilibrium systems are known to manifest interesting physical properties. It is worth mentioning that L\'evy noise-driven systems do not satisfy classical fluctuation dissipation relation. Therefore, linear response theory,  which is viewed as a generalization of the fluctuation-dissipation theorem, has been studied for SDEs driven by L\'evy noise \cite{dybiec2012fluctuation,zhang2021linear}. It is similar to the previous part that there are also some small mass limit results for SDEs driven by  L\'evy noise. For example, Talibi \cite{al2010nelson} developed Nelson theory for the $\alpha$-stable L\'evy process. Zhang \cite{zhang2008smoluchowski} obtained Smoluchowski-Kramers approximation for SDEs driven by  L\'evy noise whose moment is finite.

Hamiltonian dynamics \cite{arnol2013mathematical}, as an equivalent description of Newton's second law in the framework of classical mechanics,  form the framework of statistical mechanics. Dissipative Hamiltonian systems with noise have been investigated recently \cite{wei2019hamiltonian, wu2001large}.

In this present paper, we derive the small mass limiting equation of a class of dissipative Hamiltonian systems with L\'evy noise
\begin{equation}
\begin{aligned}
dq_t^{\e}&=\nabla_p H^{\e}(t,x_t^{\e})dt,\\
dp_t^{\e}&=(-\gamma(t,x_t^{\e})\nabla_p H^{\e}(t,x_t^{\e})-\nabla_q H^{\e}(t,x_t^{\e})+F(t,x_t^{\e}))dt+ \sigma(t,x_{t-}^{\e})dL_t,
\end{aligned}
\end{equation}
where $x_t^{\e}=(q_t^{\e},p_t^{\e})$ and $H$ is a Hamiltonian function with small mass parameter $\e$. The functions $\gamma$, $\sigma$ and $F$ are dissipation coefficient, diffusion coefficient and external force dependent on $(q_t^{\e},p_t^{\e})$, respectively. Here the process $L=\{L_t\}_{t\ge0}$ is a L\'evy process. An inspiration for this paper goes back to the work by Birrell-Wehr \cite{birrell2018homogenization}. The main idea of proof is the following: By means of the structure of Hamiltonian systems and a Lyapunov equation, we derive the limiting equation including a noise-induced drift term. Then, we prove that under appropriate assumptions, the original systems converge to the limiting equation in moment. Finally, utilizing non-explosion property of the solution of original systems, we show the convergence in probability for weaker assumptions.

This paper is organized as follows. In Section 2, we recall some basic notations and introduce a class of dissipative Hamiltonian systems with L\'evy noise. In Section 3, we state and prove the homogenization result. More precisely, in Section 3.1, we obtain the moment estimation of kinetic function and get some relevant estimation results. In Section 3.2, we derive the limiting equation by using a Lyapunov equation. In Section 3.3, we finish the proof of the main results (Theorem \ref{result1} and Theorem \ref{result2}). In Section 3.4, we extend the result to some more general systems. In Section 4, we present an illustrative example .

\section{Preliminaries}

\subsection{\textbf{L\'evy motion}}
Let $(\Omega, \mathbb{P})$ be a probability space. An stochastic process $L_t=L(t)$ taking values in $\mathbb{R}^n$ with $L(0)=0$ $a.s.$ (almost surely) is called an $n$-dimensional L\'evy process if it is stochastically continuous, with independent increments and stationary increments.

An $n$-dimensional L\'evy process $L_t$ can be expressed by L\'evy-It\^o decomposition, i.e., there exist a drift vector $b\in\Rn$, a covariance matrix $Q$ such that
\begin{equation*}
L_t=bt+B_Q(t)+\int_{||x||<1}x\widetilde N(t,dx)+\int_{||x||\ge1}x N(t,dx),
\end{equation*}
where  $N(dt,dx)$ is the Poisson random measure on $\mathbb{R}\times(\Rn\backslash\{0\})$, $\widetilde N(dt,dx)\triangleq N(dt,dx)-\nu(dx)dt$ is the compensated Poisson random measure, $\nu\triangleq\mathbb{E}N(1,\cdot)$ is the jump measure, and $B_Q(t)$ is an independent $n$-dimensional Brownian motion with covariance matrix $Q$. The triple $(b,Q,\nu)$ is called the generating triple for the L\'evy process $L_t$. A L\'evy process $L_t$ has $\th$-th moment if and only if $\int_{||x||>1}||x||^\th\nu(dx)<\infty$.

\subsection{\textbf{Dissipative Hamiltonian system with L\'evy noise}}
We consider the dissipative Hamiltonian system described in \cite{birrell2018homogenization}. Given a time-dependent Hamiltonian function $H(t,x_t)$, where $x_t=(q_t,p_t)\in\Rn\times\Rn$. The following Hamiltonian system describe a system with dissipative force and an external force.
\begin{equation}
\begin{aligned}
&\dot{q}_t=\nabla_p H(t,x_t),\\ &\dot{p}_t=-\gamma(t,x_t)\nabla_pH(t,x_t)-\nabla_qH(t,x_t)+F(t,x_t),
\end{aligned}
\end{equation}
with dissipation coefficient $\gamma:[0,\infty)\times\mathbb{R}^{2n}\to \mathbb{R}^{n\times n}$, and external forcing function $F:[0,\infty)\times\mathbb{R}^{2n}\to\Rn$.
A natural example for Hamiltonian function is $H(q,p)=\frac{p^2}{2m}+V(q)$, where $\frac{p^2}{2m}$ represents kinetic energy of system and $m$ represents mass.
Hence we are interested in a family of Hamiltonians depending on some small parameter $\e$ of the form
\begin{equation}
\label{H}
H^{\e}(t,q,p)\triangleq K^{\e}(t,q,p)+V(t,q)=K(\e,t,q,p/\sqrt{\e})+V(t,q).
\end{equation}
We remark that the notation $K$ and $V$ may not represent physical kinetic energy and potential energy. Actually, the splitting is more extensive as long as it satisfies the assumptions we will make below. However, we still call $K$ kinetic energy and $V$ potential energy function in the following sections.

In this paper, we study the following Hamiltonian system perturbed by L\'evy fluctuation
\begin{equation}
\label{slo}
\begin{aligned}
dq_t^{\e}&=\nabla_p H^{\e}(t,x_t^{\e})dt,\\
dp_t^{\e}&=(-\gamma(t,x_t^{\e})\nabla_p H^{\e}(t,x_t^{\e})-\nabla_q H^{\e}(t,x_t^{\e})+F(t,x_t^{\e}))dt+ \sigma(t,x_{t-}^{\e})dL_t,
\end{aligned}
\end{equation}
with initial data $(q_0^{\e}, \ p_0^{\e})$, where $\sigma:[0,\infty)\times\mathbb{R}^{2n}\to\mathbb{R}^{n\times d}$ is noise intensity function and $L=\{L_t\}_{t\ge0}$ is a $\mathbb{R}^d$-valued pure jump L\'evy process with triple $(0,0,\nu)$.
\begin{remark}
We consider only pure jump L\'evy process here, since by L\'evy-It\^o decomposition, L\'evy process could be expressed as a sum of a Brownian motion and a pure jump L\'evy process, in addition to a drift term which may be absorbed in the vector field in SDE. Homogenization of dissipative Hamiltonian systems with Brownian motion was studied in \cite{birrell2018homogenization}. Thereby we use same notations as in \cite{birrell2018homogenization} to make sure the influence of Brownian motion can be added to our results.
\end{remark}

We assume that the pure jump L\'evy process has finite moment. More precisely, we make the following assumption for jump measure $\nu$.\\
{\bf Assumption 1.}
There exists a constant $\th$ such that the L\'evy measure $\nu$ satisfies
\begin{equation*}
\int_{|x|\ge 1} |x|^{2\lor \th}\nu(dx)<\infty, 
\end{equation*}
here ${2\lor \th}= \max \{2,\th\}$.

\section{Homogenization of dissipative Hamiltonian systems under L\'evy fluctuations}

In this section we formulate the assumptions and state the main results Theorem \ref{result1} and Theroem \ref{result2}.
\subsection{\textbf{Moment estimates}}
In this subsection, we derive the moment estimation for kinetic energy $K$ and some relevant estimation results.
For the Hamiltonian function $H$ we make the following assumptions.\\
{\bf Assumption 2. }
The Hamiltonian function $H$ has form \eqref{H},
where $K(\e,t,q,z)$ is non-negative and $\cC^2$ in $(t,q,z)$ for each $\e$. Moreover, there exists a constant $C_0>0$ such that $K^{\e}(0,x_0^{\e})\le C_0$. For every fixed constant $T>0$ and $\e_0>0$, the following conditions hold on $(0,\e_0]\times[0,T]\times\mathbb{R}^{2n}$:\\
1. There exist positive constants $C, M_1$ such that
\begin{equation*}
\max{\{|\partial_tK(\e,t,q,z)|,||\nabla_qK(\e,t,q,z)||,||\nabla_zK(\e,t,q,z)||\}}\le M_1+CK(\e,t,q,z).
\end{equation*}
2. There exist positive constants $c, M_2$ such that
\begin{equation*}
||\nabla_zK(\e,t,q,z)||^2+M_2\ge cK(\e,t,q,z).
\end{equation*}
3. The kinetic energy $K(\e,t,q,z)$ is Lipschitz w.r.t (with respect to) $z$, i.e. there exists a constant $L$ such that
\begin{equation*}
|K(\e,t,q,z_1)-K(\e,t,q,z_2)|\le L|z_1-z_2|.
\end{equation*}
4. The potential energy $V(t,q)$ is $\mathcal{C}^1$ in $(t,q)$ and $\nabla_qV$ is bounded.

For dissipative matrix function $\gamma$, external force $F$ and noise intensity $\sigma$, we assume that\\
{\bf Assumption 3.}
For every $T>0$, the following conditions hold on $[0,T]\times\mathbb{R}^{2n}$:\\
1. The function $\gamma, F, \sigma$ are bounded and Lipschitz.\\
2. The matrix function $\gamma$ is symmetric with eignevalues bounded below by some constant $\lambda>0$.

\begin{remark}
Under the Assumption 1-3 and additional Assumption 4 below, the solution $x_t^{\e}$ to stochastic Hamiltonian system \eqref{slo} exists and is unique. See Appendix for proof.
\end{remark}
At this point, we begin to prove the moment estimations of $K$. We firstly give an upper bound of kinetic energy $K$.
\begin{lemma}
\label{L1}
For every $\th\ge 1$ and $T>0$ there exist positive constants $\alpha_0, \e_0$ such that for all constant $\alpha\in (0,\alpha_0], \epsilon\in(0,\e_0]$ and $t\in[0,T]$, we have
\begin{equation}
\label{K^q}
K^{\e}(t,x_t^\e)^\th \le \frac{\kappa(\e)}{\alpha}+\int_0^t\int_{\mathbb{R}^d\backslash \{0\}}e^{-\alpha (t-s)/\e}[K^{\e}(s,q_{s-}^{\e},p_{s-}^{\e}+\sigma(s,x_{s-}^{\e})x)^\th- K^{\e}(s,q_{s-}^{\e},p_{s-}^{\e})^\th]\widetilde N(ds,dx),
\end{equation}
where $\kappa(\e)=\kappa_1+\kappa_2\e^{1-\th/2}$ for positive constants $\kappa_1$ and $\kappa_2$.
\end{lemma}

\begin{proof}
Applying It\^o formula to $e^{\alpha t/\e}K^{\e}(t,x_t^\e)^\th$, we have

\begin{align*}
&e^{\alpha t/\e}K^{\e}(t,x_t^\e)^\th\\
&=K^{\e}(0,x_0^{\e})^\th+\frac{\alpha}{\e}\int_0^t e^{\alpha s/\e}K^{\e}(s,x_s^\e)^\th ds+ \th\int_0^t e^{\alpha s/\e}K^{\e}(s,x_s^\e)^{\th-1}(\partial_s K)^{\e}(s,x_s^{\e})ds\\
&+\frac{\th}{\e}\int_0^t e^{\alpha s/\e}K^{\e}(s,x_s^\e)^{\th-1}(\nabla_z K)^{\e}(s,x_s^{\e})(-\gamma(s,x_s^{\e}))(\nabla_z K)^{\e}(s,x_s^{\e})ds\\
&+\frac{\th}{\sqrt{\e}}\int_0^t e^{\alpha s/\e}(\nabla_z K)^{\e}(s,x_s^{\e})(-\nabla_q V(s,q_s^{\e})+F(s,x_s^{\e}))ds\\
&+\int_0^t\int_{\mathbb{R}^d\backslash \{0\}}e^{\alpha s/\e}[K^{\e}(s,q_{s-}^{\e},p_{s-}^{\e}+\sigma(s,x_{s-}^{\e})x)^\th- K^{\e}(s,q_{s-}^{\e},p_{s-}^{\e})^\th]\widetilde N(ds,dx)\\
&+\int_0^t\int_{\mathbb{R}^d\backslash \{0\}}e^{\alpha s/\e}[K^{\e}(s,q_{s-}^{\e},p_{s-}^{\e}+\sigma(s,x_{s-}^{\e})x)^\th- K^{\e}(s,q_{s-}^{\e},p_{s-}^{\e})^\th]\nu(dx)ds    \tag{$I_1$}\\
&-\int_0^t\int_{|x|<1}e^{\alpha s/\e}\sigma^i(s,x_{s-}^{\e})x\frac{\th}{\sqrt{\e}}K^{\e}(s,q_{s-}^\e,p_{s-}^\e)^{\th-1}(\nabla_{z_i}K)^{\e}(s,q_{s-}^{\e},p_{s-}^{\e}) ]\nu(dx)ds,    \tag{$I_2$}
\end{align*}
where we denote the last two integrals by $I_1, I_2$ respectively. The notation $(\partial_s K)^{\e}(s,x)$ is equal to $\partial_s K(\e,s,q,p/\sqrt{\e})$ and similarly for $(\nabla_z K)^{\e}(s,x)$.

First we estimate terms $I_1,I_2$. Using mean value theorem and Lipschitz condition of $K$ for the term $I_1$ we have
\begin{equation}
\label{L1E1}
\begin{aligned}
I_1&=\int_0^t\int_{\mathbb{R}^d\backslash \{0\}}e^{\alpha s/\e}[K^{\e}(s,q_s^{\e},p_s^{\e}+\sigma(s,x_s^{\e})x)^\th- K^{\e}(s,q_s^{\e},p_s^{\e})^\th]\nu(dx)ds\\
&\le 2^{\th-2}\th \int_0^t\int_{\mathbb{R}^d\backslash \{0\}} e^{\alpha s/\e} \left[K^{\e}(s,q_s^{\e},p_s^{\e})^{\th-1}\left|K^{\e}(s,q_s^{\e},p_s^{\e}+\sigma(s,x_s^{\e})x)-K^{\e}(s,q_s^{\e},p_s^{\e})\right|\right.\\
&\left.+\left|K^{\e}(s,q_s^{\e},p_s^{\e}+\sigma(s,x_s^{\e})x)-K^{\e}(s,q_s^{\e},p_s^{\e})\right|^{\th}\right]\nu(dx)ds\\
&\le \frac{2^{\th-2}\th L||\sigma||_{\infty}}{\sqrt{\e}}\int_{\mathbb{R}^d\backslash \{0\}}|x|\nu(dx)\int_0^t e^{\alpha s/\e} K^{\e}(s,q_s^{\e},p_s^{\e})^{2\th-1}ds+\frac{2^{\th-2}\th L^\th ||\sigma||_{\infty}}{\e^{\th/2}}\int_{\mathbb{R}^d\backslash \{0\}}|x|^\th\nu(dx)\int_0^t e^{\alpha s/\e} ds.
\end{aligned}
\end{equation}

Under Assumption 2-3, for term $I_2$ we have
\begin{equation}
\label{L1E2}
\begin{aligned}
I_2&=-\int_0^t\int_{|x|<1}e^{\alpha s/\e}\sigma^i(s,x_{s}^{\e})x\frac{\th}{\sqrt{\e}}K^{\e}(s,q_{s}^\e,p_{s}^\e)^{\th-1}(\nabla_{z_i}K)^{\e}(s,q_{s}^{\e},p_{s}^{\e}) ]\nu(dx)ds\\
&\le \frac{\th||\sigma||_{\infty}}{\e}\int_{|x|<1}|x|\nu(dx)\left(M_1\int_0^t e^{\alpha s/\e}K^{\e}(s,q_s^\e,p_s^\e)^{\th-1}ds+C\int_0^t e^{\alpha s/\e}K^{\e}(s,q_s^\e,p_s^\e)^{\th}ds\right).
\end{aligned}
\end{equation}

Then combining these two inequalities \eqref{L1E1}, \eqref{L1E2} with Assumption 2-3, we obtain
\begin{equation}\label{Ito K}
\begin{aligned}
&e^{\alpha t/\e}K^{\e}(t,x_t^\e)^\th \\
&\le K^{\e}(0,x_0^{\e})^\th +\left( \frac{\alpha}{\e}+C\th-\frac{\lambda c\th}{\e}+\frac{C\th}{\sqrt{\e}}||-\nabla_q V+F||_{\infty}\right)\int_0^t e^{\alpha s/\e}K^{\e}(s,x_s^\e)^\th ds\\
&+\th \left(M_1+\frac{\lambda M_2}{\e}+\frac{M_1}{\sqrt{\e}}||-\nabla_q V+F||_{\infty}\right)\int_0^t e^{\alpha s/\e} K^{\e}(s,x_s^\e)^{\th-1} ds\\
&+ \left(\frac{2^{\th-2}\th L||\sigma||_{\infty}}{\sqrt{\e}}\int_{\mathbb{R}^d\backslash \{0\}}|x|\nu(dx)+\frac{\th||\sigma||_{\infty}}{\e}\int_{|x|<1}|x|\nu(dx)\right)\int_0^t e^{\alpha s/\e} K^{\e}(s,x_s^\e)^{\th-1} ds\\
&+\frac{C\th||\sigma||_{\infty}}{\e}\int_0^t e^{\alpha s/\e}K^{\e}(s,q_s,p_s)^\th ds+\frac{2^{\th-2}\th L^\th ||\sigma||_{\infty}^\th}{\e^{\th/2}}\int_{\mathbb{R}^d\backslash \{0\}}|x|^\th\nu(dx)\int_0^t e^{\alpha s/\e} ds\\
&+\int_0^t\int_{\mathbb{R}^d\backslash \{0\}}e^{\alpha s/\e}[K^{\e}(s,q_{s-}^{\e},p_{s-}^{\e}+\sigma(s,x_{s-}^{\e})x)^\th- K^\e(s,q_{s-}^{\e},p_{s-}^{\e})^\th]\widetilde N(ds,dx).
\end{aligned}
\end{equation}

Note that Young inequality allows $K^{\th-1}\le\frac{1}{\th}\left(\frac{M}{\delta}\right)^{\th-1}+\frac{\delta}{M}K^{\th}$. Let $M=\max\{M_1,M_2\}$. We get
\begin{equation}
\begin{aligned}
K^{\e}(t,x_t^{\e})^{\th} &\le e^{-\alpha t/\e}K^{\e}(0,x_0^{\e})-\frac{D}{\e} \int_0^t e^{-\alpha (t-s)/\e}K^{\e}(s,x_s^{\e})^\th ds+ \frac{d}{\alpha}\\
&+\int_0^t\int_{\mathbb{R}^d\backslash \{0\}}e^{-\alpha (t-s)/\e}[K^{\e}(s,q_{s-}^{\e},p_{s-}^{\e}+\sigma(s,x_{s-}^{\e})x)^\th- K^{\e}(s,q_{s-}^{\e},p_{s-}^{\e})^\th]\widetilde N(ds,dx),
\end{aligned}
\end{equation}
where
\begin{equation}
\begin{aligned}
D=&\lambda c\th-\alpha-C\th\e-C\th\sqrt{\e}||-\nabla_q V+F||_{\infty}-\th\delta\e-\th\delta\lambda-\th\delta\sqrt{\e}||-\nabla_q V+F||_{\infty}\\
&-2^{\th-2}\th L\delta||\sigma||_{\infty}M^{-1}\sqrt{\e}\int_{\mathbb{R}^d\backslash \{0\}}|x|\nu(dx)-\th\delta||\sigma||_{\infty}\int_{|x|<1}|x|\nu(dx)-C\th\delta||\sigma||_{\infty},
\end{aligned}
\end{equation}
and
\begin{equation}
\begin{aligned}
d&=\left(\frac{M}{\delta}\right)^{\th-1}\left(M\e+\lambda M+ M\sqrt{\e}||-\nabla_q V+F||_{\infty}+2^{\th-2}L\sqrt{\e}||\sigma||_{\infty}\int_{\mathbb{R}^d\backslash \{0\}}|x|\nu(dx)+ ||\sigma||_{\infty}\int_{|x|<1}|x|\nu(dx)\right)\\
&+\left(\frac{M}{\delta}\right)^{\th-1}2^{\th-2}L^\th||\sigma||_{\infty}^\th\e^{1-\th/2}\int_{\mathbb{R}^d\backslash \{0\}}|x|^\th\nu(dx).
\end{aligned}
\end{equation}
For all $\e,\delta,\alpha$ sufficiently small, $D$ is non-negative. In addition, $K^{\e}(0,x_0^{\epsilon})$ is bounded by Assumption 2. Thus we obtain the required inequality \eqref{K^q}.
\end{proof}

Now we give the moment estimation of the kinetic energy $K^{\e}(t,x_t^{\e})$ by means of above assumptions and lemma.
\begin{lemma}
\label{supE}
{\bf (Supremum of expectation of the kinetic energy)}
Under Assumption 1-3, for every positive $T$ and $\th$, the kinetic energy $K$ has the following uniform estimate
\begin{equation}
\label{supEK}
\sup_{t\in[0,T]}\E\left[K^{\e}(t,x_t^\e)^\th\right]=O(\e^{1-\frac{2\lor\th}{2}}), \ \text{as} \ \e \to 0.
\end{equation}
\end{lemma}

\begin{proof}
We first consider $\th\ge1$. Note that
\begin{equation*}
\int_0^t\int_{\mathbb{R}^d\backslash \{0\}}e^{-\alpha (t-s)/\e}[K^{\e}(s,q_{s-}^{\e},p_{s-}^{\e}+\sigma(s,x_s^{\e})x)^\th- K^{\e}(s,q_{s-}^{\e},p_{s-}^{\e})^\th]\widetilde N(ds,dx)
\end{equation*}
is a local martingale and it is in fact a martingale by using appropriate sequence of stopping times (see \cite{applebaum2009levy}, page 266).
Then we obtain the following equality
\begin{equation*}
\E \left[\int_0^t\int_{\mathbb{R}^d\backslash \{0\}}e^{-\alpha (t-s)/\e}[K^{\e}(s,q_{s-}^{\e},p_{s-}^{\e}+\sigma(s,x_s^{\e})x)^\th- K^{\e}(s,q_{s-}^{\e},p_{s-}^{\e})^\th]\widetilde N(ds,dx)\right]=0.
\end{equation*}
It follows that the equality \eqref{supEK} holds from Lemma \ref{L1} and preceding equation for $\th\ge1$.
The results for $0<\th<1$ follows by H\"older's inequality.
\end{proof}

\begin{lemma}
\label{Esup}
{\bf (Expectation of supremum of the kinetic energy)}
Under Assumption 1-3 and for every positive $T$ and $\th$, the kinetic energy $K$ has the following uniform estimate
\begin{equation}\label{Esupeq}
\E\left[\sup_{t\in[0,T]}K^{\e}(t,x_t^{\e})^\th\right]= O(\e^{-\frac{\th}{2}}), \ as\ \e\to 0.
\end{equation}
\end{lemma}

\begin{proof}
By Lemma \ref{L1} we have
\begin{equation}
\label{T1E1}
K^{\e}(t,x_t^{\e}) \le \frac{\kappa}{\alpha}+ \int_0^t\int_{\mathbb{R}^d\backslash \{0\}}e^{-\alpha (t-s)/\e}[K^{\e}(s,q_{s-}^{\e},p_{s-}^{\e}+\sigma(s,x_{s-}^{\e})x)- K^{\e}(s,q_{s-}^{\e},p_{s-}^{\e})]\widetilde N(ds,dx).
\end{equation}
It\^o's product formula implies that
\begin{equation}
\label{T1E2}
\begin{aligned}
&\int_0^t\int_{\mathbb{R}^d\backslash \{0\}}e^{-\alpha (t-s)/\e}[K^{\e}(s,q_{s-}^{\e},p_{s-}^{\e}+\sigma(s,x_{s-}^{\e})x)- K^{\e}(s,q_{s-}^{\e},p_{s-}^{\e})]\widetilde N(ds,dx)\\
=&\int_0^t\int_{\mathbb{R}^d\backslash \{0\}} [K^{\e}(s,q_{s-}^{\e},p_{s-}^{\e}+\sigma(s,x_{s-}^{\e})x)- K^{\e}(s,q_{s-}^{\e},p_{s-}^{\e})]\widetilde N(ds,dx)\\
&+\int_0^t\frac{\alpha}{\e}e^{-\alpha (t-s)/\e} \int_0^s \int_{\mathbb{R}^d\backslash \{0\}} [K^{\e}(r,q_{r-}^{\e},p_{r-}^{\e}+\sigma(r,x_{r-}^{\e})x)- K^{\e}(r,q_{r-}^{\e},p_{r-}^{\e})]\widetilde N(dr,dx) ds.
\end{aligned}
\end{equation}

We first show the proposition in the case when $\th \ge 2$.
Substituting \eqref{T1E2} into \eqref{T1E1} and taking supremum and expectation on both side, we have
\begin{equation}
\label{T1E3}
\begin{aligned}
&\E\left[\sup_{t\in[0,T]}K^{\e}(t,x_t^{\e})^\th\right]\\
\le & 2^{\th-1}\left(\frac{\kappa}{\alpha}\right)^\th+ 4^{\th-1}\E\left[\sup_{t\in[0,T]}\left|\int_0^t\int_{\mathbb{R}^d\backslash \{0\}} [K^{\e}(s,q_{s-}^{\e},p_{s-}^{\e}+\sigma(s,x_{s-}^{\e})x)- K^{\e}(s,q_{s-}^{\e},p_{s-}^{\e})]\widetilde N(ds,dx)\right|^\th\right]\\
+& 4^{\th-1}\E\left[\sup_{t\in[0,T]}\left| \int_0^t\frac{\alpha}{\e}e^{-\alpha (t-s)/\e} \int_0^s \int_{\mathbb{R}^d\backslash \{0\}} [K^{\e}(r,q_{r-}^{\e},p_{r-}^{\e}+\sigma(r,x_{r-}^{\e})x)- K^{\e}(r,q_{r-}^{\e},p_{r-}^{\e})]\widetilde N(dr,dx) ds \right|^\th\right].
\end{aligned}
\end{equation}

For the first Poisson stochastic integral term, Kunita first inequality (\cite{applebaum2009levy}, Theorem 4.4.23) implies that
\begin{equation}
\label{T1E4}
\begin{aligned}
&\E\left[\sup_{t\in[0,T]}\left|\int_0^t\int_{\mathbb{R}^d\backslash\{0\}}K^{\e}(s,q_{s-}^{\e},p_{s-}^{\e}+\sigma(s,x_{s-}^{\e})x)- K^{\e}(s,q_{s-}^{\e},p_{s-}^{\e})\widetilde N(ds,dx)\right|^\th\right]\\
\le &D(\th)\E\left[\left(\int_0^T\int_{\mathbb{R}^d\backslash\{0\}}|K^{\e}(s,q_{s-}^{\e},p_{s-}^{\e}+\sigma(s,x_{s-}^{\e})x)- K^{\e}(s,q_{s-}^{\e},p_{s-}^{\e})|^2\nu(dx)ds\right)^{\frac{\th}{2}}\right]\\
+&\E\left[\int_0^T\int_{\mathbb{R}^d\backslash\{0\}}|K^{\e}(s,q_{s-}^{\e},p_{s-}^{\e}+\sigma(s,x_{s-}^{\e})x)- K^{\e}(s,q_{s-}^{\e},p_{s-}^{\e})|^\th\nu(dx)ds\right]\\
\le &D(\th)\e^{-\frac{\th}{2}}T^{\frac{\th}{2}}L^\th||\sigma||_{\infty}^\th\left(\int_{\mathbb{R}^d\backslash
\{0\}}|x|^2\nu(dx)\right)^{\frac{\th}{2}}+\e^{-\frac{\th}{2}}TL^\th||\sigma||_{\infty}^\th\int_{\mathbb{R}^d\backslash \{0\}}|x|^\th
\nu(dx)\\
=&O(\e^{-\frac{\th}{2}}).
\end{aligned}
\end{equation}
Next we deal with the second Poisson stochastic integral term
\begin{equation}
\label{T1E5}
\begin{aligned}
& \E\left[\sup_{t\in[0,T]}\left| \int_0^t\frac{\alpha}{\e}e^{-\alpha (t-s)/\e} \int_0^s \int_{\mathbb{R}^d\backslash \{0\}} [K^{\e}(r,q_{r-}^{\e},p_{r-}^{\e}+\sigma(r,x_{r-}^{\e})x)- K^{\e}(r,q_{r-}^{\e},p_{r-}^{\e})]\widetilde N(dr,dx) ds \right|^\th\right]\\
\le &\E\left[\sup_{t\in[0,T]}\left|\int_0^t \frac{\alpha}{\e}e^{-\alpha (t-s)/\e} \sup_{s\in[0,t]} \left|\int_0^s \int_{\mathbb{R}^d\backslash \{0\}} [K^{\e}(r,q_{r-}^{\e},p_{r-}^{\e}+\sigma(r,x_{r-}^{\e})x)- K^{\e}(r,q_{r-}^{\e},p_{r-}^{\e})]\widetilde N(dr,dx)\right| ds\right|^\th\right]\\
\le& \E\left[\sup_{t\in[0,T]}\left|\int_0^t\int_{\mathbb{R}^d\backslash\{0\}}K^{\e}(s,q_{s-}^{\e},p_{s-}^{\e}+\sigma(s,x_{s-}^{\e})x)- K^{\e}(s,q_{s-}^{\e},p_{s-}^{\e})\widetilde N(ds,dx)\right|^\th\right]\\
=&O(\e^{-\frac{\th}{2}}),
\end{aligned}
\end{equation}
where the last equality is obtained by utilizing \eqref{T1E4}. Therefore, equality \eqref{Esupeq} holds for $\th\ge2$ by \eqref{T1E3}, \eqref{T1E4}and \eqref{T1E5}. It follows for all $\th>0$ by H\"older's inequality.

\end{proof}

We make an additional assumption for kinetic energy $K$ as follows.\\
{\bf Assumption 4 }
For every $T>0$,  there exist $c>0, \eta>0$ such that
\begin{equation*}
K(\e,t,q,z)\ge c||z||^{\eta}.
\end{equation*}
Now we can deduce an useful proposition under this assumption. Proposition \ref{proposition1} is a direct deduction from Lemma \ref{supE}, Lemma \ref{Esup} and Assumption 4.
\begin{proposition}
\label{proposition1}
Under Assumption 1-4, for every $T>0$ we have
\begin{equation}
\label{supEp}
\sup_{t\in[0,T]}\E\left[||p_t^\e||^\th\right]=\left\{
\begin{aligned}
&O(\e^{\frac{\th}{2}}), \qquad \text{if} \ \th\le2\eta, \\
&O(\e^{\frac{\th}{2}+1-\frac{\th}{2\eta}}), \qquad \text{if} \ \th>2\eta,
\end{aligned}
\right. \
\text{as} \ \e \to 0,
\end{equation}
and
\begin{equation}
\label{Esupp}
\E\left[\sup_{t\in[0,T]}||p_t^{\e}||^\th\right]=O(\e^{\frac{\th}{2}-\frac{\th}{2\eta}}), \ \text{as} \ \e\to 0.
\end{equation}
\end{proposition}

\begin{proof}
From Assumption 4, we have
\begin{equation*}
\sup_{t\in[0,T]}\E\left[||p_t^\e||^\th\right] \le \e^{\frac{\th}{2}}\sup_{t\in[0,T]}\E\left[K^{\e}(t,x_t^\e)^{\frac{\th}{\eta}}\right].
\end{equation*}
Note that Lemma \ref{supE} implies $\sup_{t\in[0,T]}\E\left[K^\e(t,x_t^\e)^a\right]=O(1)$ for $a\le2$ and $\sup_{t\in[0,T]}\E\left[K^\e(t,x_t^\e)^a\right]=O(\e^{1-\frac{a}{2}})$ for $a>2$. Hence we get \eqref{supEp}. Equation \eqref{Esupp} follows similar arguments and Lemma \ref{Esup}.
\end{proof}

\begin{remark}
If the parameter $\eta$ in Assumption 4 was given, then proposition 1 told us the order of momentum $p_t^\e$ convergence to zero. For example, assume that $\eta$ in Assumption 4 equals to 2, we have $\sup_{t\in[0,T]}\E\left[||p_t^\e||^\th\right]=O(\e^{\frac{\th}{2}})$ when $\th\le 4$ and $\sup_{t\in[0,T]}\E\left[||p_t^\e||^\th\right]=O(\e^{1+\frac{\th}{4}})$ when $\th>4$. Moreover, $\E\left[\sup_{t\in[0,T]}||p_t^\e||^\th\right]=O(\e^{\frac{\th}{4}})$.
\end{remark}

\subsection{\textbf{Derivation of the limit equation}}
In this subsection, we derive the limit equation of the system \eqref{slo} as $\e\to 0$.
To this end we make an additional assumption on $\gamma$.\\
{\bf Assumption 5 }
Every element $\gamma_i^j$ in matrix function $\gamma$ is $C^1$ and independent of $p$.

Note that stochastic Hamiltonian equation \eqref{slo} can be simplified to
\begin{equation}
\label{simplified1}
\begin{aligned}
d(q_t^{\e})&=\nabla_pH^{\e}(t,x_t^{\e})dt\\
&=\gamma^{-1}(t,x_t^{\e})(\nabla_qH^{\e}(t,x_t^{\e})-F(t,x_t^{\e}))dt +\gamma^{-1}(t,x_t^{\e})\sigma(t,x_{t-}^{\e})dL_t-\gamma^{-1}(t,x_t^{\e})d(p_t^{\e}).
\end{aligned}
\end{equation}
Since matrix function $\gamma$ has bounded eigenvalues, $\gamma$ is invertible.
Taking stochastic integration by parts formula for the last term $\gamma^{-1}(t,x_t^{\e})d(p_t^{\e})$ on the right hand side of \eqref{simplified1}, we have
\begin{equation*}
\begin{aligned}
(\gamma^{-1})_i^j(t,q_t^{\e})d(p_t^{\e})_j=&-d((\gamma^{-1})_i^j(t,q_t^{\e})(p_t^{\e})_j) +(p_{t-}^{\e})_j\partial_t(\gamma^{-1})_i^j(t,q_t^{\e})dt\\
&+(p_{t-}^{\e})_j\partial_{q^l}(\gamma^{-1})_i^j(t,q_t^{\e})\partial_{p_l}H^{\e}(t,x_t^{\e})dt,
\end{aligned}
\end{equation*}
where $\partial_{q^l}(\gamma^{-1})_i^j$ means the $l$-th component of $\nabla_q(\gamma^{-1})_i^j$, and $\partial_{p_l}H$ means the $l$-th component of $\nabla_qH$. Here we used Einstein summation notation.
Therefore,
\begin{equation}
\label{D1}
\begin{aligned}
d(q_t^{\e})_i=&(\gamma^{-1})_i^j(t,q_t^{\e})(\partial_{q_j}H^{\e}(t,x_t^{\e})-F_j(t,x_t^{\e}))dt +(\gamma^{-1})_i^j(t,q_t^{\e})\sigma_j^{\rho}(t,x_{t-}^{\e})d(L_t)_{\rho}\\
&-d((\gamma^{-1})_i^j(t,q_t^{\e})(p_t^{\e})_j) +(p_{t-}^{\e})_j\partial_t(\gamma^{-1})_i^j(t,q_t^{\e})dt +(p_{t-}^{\e})_j\partial_{q^l}(\gamma^{-1})_i^j(t,q_t^{\e})\partial_{p_l}H^{\e}(t,x_t^{\e})dt.
\end{aligned}
\end{equation}

To simplify the last term $(p_t^{\e})_j\partial_{p_l}H^{\e}(t,x_t^{\e})dt$, we compute
\begin{equation}
\begin{aligned}
&d((p_t^{\e})_i(p_t^{\e})_j)=(p_{t-}^{\e})_id(p_t^{\e})_j+(p_{t-}^{\e})_jd(p_t^{\e})_i+d[p_i^{\e},p_j^{\e}]_t\\
=&(p_{t-}^{\e})_i\left[(-\gamma_j^k(t,p_t^{\e})\partial_{p_k} H^{\e}(t,x_t^{\e})-\partial_{q_j} H^{\e}(t,x_t^{\e})+F_j(t,x_t^{\e}))dt+ \sigma_j^{\rho}(t,x_{t-}^{\e})d(L_t)_{\rho}\right]\\
+&(p_{t-}^{\e})_j\left[(-\gamma_i^k(t,p_t^{\e})\partial_{p_k} H^{\e}(t,x_t^{\e})-\partial_{q_i} H^{\e}(t,x_t^{\e})+F_i(t,x_t^{\e}))dt+ \sigma_i^{\rho}(t,x_{t-}^{\e})d(L_t)_{\rho}\right]\\
+&\int_{\mathbb{R}^d \backslash\{0\}}\sigma_i^k(t,x_{t-}^{\e})\sigma_j^l(t,x_{t-}^{\e})x_kx_lN(dt,dx).
\end{aligned}
\end{equation}
Rewrite this equation in the form of the following Lyapunov equation \cite{ortega2013matrix}
\begin{equation}
\label{Lyapunoveq}
\gamma_j^k(V_t)_{ki}+\gamma_i^k(V_t)_{kj}=(C_t)_{ij},
\end{equation}
where $(V_t)_{ij}=\partial_{p_i} H^{\varepsilon}(t,x_t^{\varepsilon})(p_{t-}^{\varepsilon})_jdt$, and
\begin{equation*}
\begin{aligned}
(C_t)_{ij}=&-d((p_t^{\e})_i(p_t^{\e})_j)+(p_{t-}^{\e})_i\left[-\partial_{q_j} H^{\e}(t,x_t^{\e})+F_j(t,x_t^{\e})\right]dt+(p_{t-}^{\e})_j\left[-\partial_{q_i} H^{\e}(t,x_t^{\e})+F_i(t,x_t^{\e})\right]dt\\
+&(p_{t-}^{\e})_i\sigma_j^{\rho}(t,x_{t-}^{\e})d(L_t)_{\rho}+(p_{t-}^{\e})_j\sigma_i^{\rho}(t,x_{t-}^{\e})d(L_t)_{\rho}+ \int_{\mathbb{R}^d\backslash\{0\}}\sigma_i^k(t,x_{t-}^{\e})\sigma_j^l(t,x_{t-}^{\e})x_kx_l N(dt,dx).
\end{aligned}
\end{equation*}
By solving Lyapunov equation \eqref{Lyapunoveq}, we have
\begin{equation*}
(V_t)_{ij}=\int_0^{\infty}e^{-y\gamma_i^k}(C_t)_{kl}e^{-y\gamma_j^l}dy.
\end{equation*}
Hence, we have
\begin{equation}
\label{D2}
\begin{aligned}
&(p_{t-}^{\e})_j\partial_{p_l}H^{\e}(t,x_t^{\e})dt=G_{jl}^{ab}(t,q_t^{\e})(C_t)_{ab}\\
&=G_{jl}^{ab}(t,q_t^{\e})\left[-d((p_t^{\e})_a(p_t^{\e})_b) +(p_{t-}^{\e})_a(-\partial_{p_b}H^{\e}(t,x_t^{\e})+F_b(t,x_t^{\e}))dt\right.\\
&\left. +(p_{t-}^{\e})_b(-\partial_{p_a}H^{\e}(t,x_t^{\e})+F_a(t,x_t^{\e}))dt+(p_{t-}^{\e})_a\sigma_b^{\rho}(t,x_{t-}^{\e})d(L_t)_{\rho} +(p_{t-}^{\e})_b\sigma_a^{\rho}(t,x_{t-}^{\e})d(L_t)_{\rho}\right.\\
&+\int_{\mathbb{R}^d\backslash\{0\}}\sigma_a^k(t,x_{t-}^{\e})\sigma_b^l(t,x_{t-}^{\e})x_kx_l N(dt,dx)],
\end{aligned}
\end{equation}
where $G_{jl}^{ab}(t,q_t^{\e})=\int_0^{\infty}e^{-y\gamma_j^a(t,q_t^{\e})}e^{-y\gamma_l^b(t,q_t^{\e})}dy$.

Combining Eq.\eqref{D1} and Eq.\eqref{D2} together, we see that $q_t^{\e}$ satisfies the equation
\begin{equation}
\label{original}
\begin{aligned}
d(q_t^{\e})_i&=(\gamma^{-1})_i^j(t,q_t^{\e})\left(\partial_{q_j}V(t,q_t^{\e})+F_j(t,x_t^{\e})\right)dt +(\gamma^{-1})_i^j(t,q_t^{\e})\sigma_j^{\rho}(t,x_{t-}^{\e})d(L_t)_{\rho}\\
&+(\gamma^{-1})_i^j(t,q_t^{\e})\partial_{q_j}K^{\e}(t,x_t^{\e})dt-\partial_{q^h}(\gamma^{-1})_i^j(t,q_t^{\e}) G_{jh}^{ab}(t,q_t^{\e})\int_{\mathbb{R}^d\backslash\{0\}}\sigma_a^k(t,x_{t-}^{\e})\sigma_b^l(t,x_{t-}^{\e})x_kx_l N(dt,dx)\\
&+d(R_t^{\e})_i,
\end{aligned}
\end{equation}
where
\begin{equation}
\label{error}
\begin{aligned}
d(R_t^{\e})_i&=d((\gamma^{-1})_i^j(t,q_t^{\e})(p_t^{\e})_j)-(p_t^{\e})_j\partial_t(\gamma^{-1})_i^j(t,q_t^{\e})dt\\ &-\partial_{q^h}(\gamma^{-1})_i^j(t,q_t^{\e})G_{jh}^{ab}(t,q_t^{\e})\left[-d((p_t^{\e})_a(p_t^{\e})_b) +(p_{t-}^{\e})_a(-\partial_{p_b}H^{\e}(t,x_t^{\e})+F_b(t,x_t^{\e}))dt\right.\\
&+\left.(p_{t-}^{\e})_b(-\partial_{p_a}H^{\e}(t,x_t^{\e})+F_a(t,x_t^{\e}))dt+(p_{t-}^{\e})_a\sigma_b^{\rho}(t,x_{t-}^{\e})d(L_t)_{\rho} +(p_{t-}^{\e})_b\sigma_a^{\rho}(t,x_{t-}^{\e})d(L_t)_{\rho}\right].
\end{aligned}
\end{equation}
Note that term $(\gamma^{-1})_i^j(t,q_t^{\e})\partial_{q_j}K^{\e}(t,x_t^{\e})dt$ in \eqref{original} will survive in the limiting equation. Here we make another assumption. \\
{\bf Assumption 6 }
Every element $\partial_{q_j}K$ in $\nabla_qK$ is Lipschitz w.r.t $q$.
\begin{remark}
This assumption seems a little strong. However, it is reasonable since we assume function $K$ is $\mathcal{C}^2$, hence $K$ is locally Lipschitz. Indeed we will extend our results to locally Lipshitz $K$ in Section 3.4. If $K$ is independent of $q$, then this term can be ignored. If $K$ does not have additional assumption, we refer to \cite{birrell2019homogenization} for estimations of this term.
\end{remark}
The proceeding calculations motivate the proposed lower dimensional limiting equation for the dynamics of position $q$:
\begin{equation}
\label{limiting}
\begin{aligned}
d(q_t)_i&=(\gamma^{-1})_i^j(t,q_t)\left(\partial_{q_j}V(t,q_t)+F_j(t,x_t)\right)dt +(\gamma^{-1})_i^j(t,q_t)\sigma_j^{\rho}(t,x_{t-})d(L_t)_{\rho}\\
&+(\gamma^{-1})_i^j(t,q_t)\partial_{q_j}K(t,x_t)dt -\partial_{q^h}(\gamma^{-1})_i^j(t,q_t) G_{jh}^{ab}(t,q_t)\int_{\mathbb{R}^d\backslash\{0\}}\sigma_a^k(t,x_{t-})\sigma_b^l(t,x_{t-})x_kx_l N(dt,dx),
\end{aligned}
\end{equation}
where $x_t=(q_t,0)$ since momentum $p_t^\e$ converges to $0$ from Proposition \ref{proposition1}. Here we denote
\begin{equation}
S_i(t,x)=\int_0^t \partial_{q^h}(\gamma^{-1})_i^j(t,q)G_{jh}^{ab}(s,q)\int_{\mathbb{R}^d\backslash\{0\}}\sigma_a^k(s,x_{s-})\sigma_b^l(s,x_{s-})z_kz_lN(ds,dz).
\end{equation}
Actually it is the noise-induced drift in limiting equation.

\subsection{\textbf{Proof of convergence to the limiting equation}}
In this subsection, we show that the stochastic Hamiltonian system \eqref{slo} converge to homogenized equation \eqref{limiting} in moment under an additional assumption:\\
{\bf Assumption 7.}
Assume that function $\gamma$ is $\mathcal{C}^2$ and $\partial_t \gamma$, $\partial_{q^i} \gamma$, $\partial_t\partial{q^i} \gamma$ and $\partial_{q^i}\partial_{q^j} \gamma$ are bounded on $[0,T]\times\mathbb{R}^n$, for every $T$.

Now we demonstrate that the remainder term $R_t^{\e}$ converges to zero. For convenience, we denote $\tilde{C}$ a finite positive constant whose value may vary from line to line and the notation $\tilde{C}(\cdot)$ to emphasize the dependence on the quantities appearing in the parentheses.
\begin{lemma}
\label{R}
Under Assumption 1-7, for every $T>0, \eta>1$ and $\th<\eta$, we have
\begin{equation}
\E\left[\sup_{t\in[0,T]}||R_t^{\e}||^\th\right]=O(\e^\beta),\ as \ \e\to 0,
\end{equation}
where $R_t^{\e}$ was defined in Eq. \eqref{error} and $\beta(\th)$ is a piecewise function
\[\beta(\th)=
\begin{cases}
\frac{\th}{2}\left(1-\frac{1}{\eta}\right),  & 0<\th\le\frac{2\eta}{\eta+1},\\
1-\frac{\th}{\eta},  & \th>\frac{2\eta}{\eta+1}.
\end{cases}\]

\end{lemma}

\begin{proof}
Integrating Eq. \eqref{error} on $[0,T]$, then taking expectation and supremum on it, we have
\begin{equation*}
\begin{aligned}
&\E\left[\sup_{t\in[0,T]}||R_t^{\e}||^\th\right]\le 8^{\th-1}\left( \E\left[\sup_{t\in[0,T]} ||(\gamma^{-1})_i^j(t,q_t^{\e})(p_t^{\e})_j||^\th\right]+\E\left[\sup_{t\in[0,T]} ||(\gamma^{-1})_i^j(0,q_0^{\e})(p_0^{\e})_j||^\th\right]\right.\\
&+\E\left[\sup_{t\in[0,T]}\left|\left|\int_0^t(p_s^{\e})_j\partial_s(\gamma^{-1})_i^j(s,q_s^{\e})ds\right|\right|^\th\right] +\E\left[\sup_{t\in[0,T]}\left|\left|\int_0^t\partial_{q^h}(\gamma^{-1})_i^j(s,q_s^{\e})G_{jh}^{ab}(s,q_s^{\e})d((p_s^{\e})_a(p_s^{\e})_b)\right|\right|^\th\right]\\
&+\E\left[\sup_{t\in[0,T]}\left|\left|\int_0^t \partial_{q^h}(\gamma^{-1})_i^j(s,q_s^{\e})G_{jh}^{ab}(s,q_s^{\e})(p_{s}^{\e})_a\left(\partial_{q_b}K^{\e}(s,x_s^{\e})+\partial_{q_b}V(s,q_s^{\e})+F_b(s,x_s^{\e})ds\right)\right|\right|^\th\right]\\
&+\E\left[\sup_{t\in[0,T]}\left|\left|\int_0^t \partial_{q^h}(\gamma^{-1})_i^j(s,q_s^{\e})G_{jh}^{ab}(s,q_s^{\e})(p_{s}^{\e})_b\left(\partial_{q_a}K^{\e}(s,x_s^{\e})+\partial_{q_a}V(s,q_s^{\e})+F_a(s,x_s^{\e})ds\right)\right|\right|^\th\right]\\
&+\E\left[\sup_{t\in[0,T]}\left|\left|\int_0^t \partial_{q^h}(\gamma^{-1})_i^j(t,q_t^{\e})G_{jh}^{ab}(t,q_t^{\e})(p_{t-}^{\e})_a\sigma_b^{\rho}(t,x_t^{\e})d(L_t)_{\rho}\right|\right|^\th\right]\\
&+\E\left[\sup_{t\in[0,T]}\left|\left|\int_0^t \partial_{q^h}(\gamma^{-1})_i^j(t,q_t^{\e})G_{jh}^{ab}(t,q_t^{\e})(p_{t-}^{\e})_b\sigma_a^{\rho}(t,x_t^{\e})d(L_t)_{\rho}\right|\right|^\th\right]\\
&:=\sum_{i=1}^8 J_i.
\end{aligned}
\end{equation*}
We will now give upper bounds of terms $\{J_i\}_{i=1}^8$ for $\th\ge1$.
For the first two terms,
\begin{equation*}
J_1+J_2\le 2||\gamma^{-1}||_{\infty}^\th\E\left[\sup_{t\in[0,T]}||p_t^{\e}||^\th\right].
\end{equation*}
For the third term, we have
\begin{equation*}
J_3\le T^{\th-1}||\partial_t{\gamma^{-1}}||_{\infty}^\th \E\left[\int_0^T||p_s^{\e}||^\th ds\right]
\le T^\th||\partial_t{\gamma^{-1}}||_{\infty}^\th \sup_{t\in[0,T]}\E\left[||p_t^{\e}||^\th\right].
\end{equation*}
Note by Assumption 7 we can deduce that the function $\partial_q(\gamma^{-1})(t,q)G(t,q)$ is bounded and $\mathcal{C}^1$. Hence we have the following estimation (see Appendix)
\begin{equation}
\label{J4}
J_4\le \tilde{C}(\th,T,M_1,C,\gamma)\left(\E[\sup_{t\in[0,T]}||p_t^{\e}||^{2\th}]+\e^{-\frac{\th}{2}}\sup_{t\in[0,T]}\E \left[||p_t^\e||^{2\th}\right] +\e^{-\frac{\th}{2}}\sup_{t\in[0,T]}\E\left[||p_t^\e||^{2\th}K^\e(t,x_t^\e)^\th\right]\right).
\end{equation}

Applying H\"older inequality and Assumption 2-3 we have
\begin{equation*}
\begin{aligned}
J_5&\le T^{\th-1}
\E\left[\sup_{t\in[0,T]}\int_0^t||p_s^{\e}||^\th\left(\left|\left|\nabla_qK^{\e}(s,x_s^{\e})\right|\right|^\th+||\nabla_qV+F||_{\infty}^\th\right)ds\right]\\
&\le T^\th\left(\sup_{t\in[0,T]}\E\left[||p_t^{\e}||^\th||K^{\e}(t,x_t^{\e})||^\th\right]+(M_1^\th+||\nabla_qV+F||_{\infty}^\th)\sup_{t\in[0,T]}\E\left[||p_t^{\e}||^\th\right]\right).
\end{aligned}
\end{equation*}
The estimation of $J_6$ is similar to $J_5$.
For the last two term (see Appendix), we have
\begin{equation}
\label{J7}
J_7\le \tilde{C}(\th,T,\nu)\sup_{t\in[0,T]}\E\left[||p_t^{\e}||^\th\right].
\end{equation}
The estimation of $J_8$ is similar to $J_7$ as well. Substitute all these upper bound together, we obtain
\begin{equation*}
\begin{aligned}
\E\left[\sup_{t\in[0,T]}||R_t^{\e}||^\th\right]&\le \tilde{C}\left(\E\left[\sup_{t\in[0,T]}||p_t^\e||^\th\right] +\E\left[\sup_{t\in[0,T]}||p_t^\e||^{2\th}\right] +\sup_{t\in[0,T]}\E\left[||p_t^\e||^\th\right] +\e^{-\frac{\th}{2}}\sup_{t\in[0,T]}\E\left[||p_t^\e||^{2\th}\right]\right.\\
&\left.+\sup_{t\in[0,T]}\E\left[||p_t^\e||^\th K^\e(t,x_t^\e)^\th\right] +\e^{-\frac{\th}{2}}\sup_{t\in[0,T]}\E\left[||p_t^\e||^{2\th} K^\e(t,x_t^\e)^\th\right]\right)\\
&\le \tilde{C}\left(\E\left[\sup_{t\in[0,T]}||p_t^\e||^\th\right] +\E\left[\sup_{t\in[0,T]}||p_t^\e||^{2\th}\right] +\sup_{t\in[0,T]}\E\left[||p_t^\e||^\th\right] \right)\\
&+\tilde{C}\e^{\frac{\th}{2}}\left(\sup_{t\in[0,T]}\E\left[K^\e(t,x_t^\e)^{\th+\frac{\th}{\eta}}\right] +\sup_{t\in[0,T]}\E\left[K^\e(t,x_t^\e)^{\frac{2\th}{\eta}}\right] +\sup_{t\in[0,T]}\E\left[K^\e(t,x_t^\e)^{\th+\frac{2\th}{\eta}}\right]\right).
\end{aligned}
\end{equation*}
The last inequality follows from the similar arguments in proposition \ref{proposition1}. Now we only need to compare order of $\e$ in these terms. By means of Lemma \ref{supE} and Proposition \ref{proposition1}, we obtain
\begin{equation}
\E\left[\sup_{t\in[0,T]}||R_t^{\e}||^\th\right]=O(\e^{\frac{\th}{2}\left(1-\frac{1}{\eta}\right)})+O(\e^{1-\frac{\th}{\eta}}).
\end{equation}
Thus if $\th>2-\frac{2}{\eta+1}$, then $\E\left[\sup_{t\in[0,T]}||R_t^{\e}||^\th\right]=O(\e^{1-\frac{\th}{\eta}})$. If $1\le\th\le2-\frac{2}{\eta+1}$ then $\E\left[\sup_{t\in[0,T]}||R_t^{\e}||^\th\right]=O(\e^{\frac{\th}{2}\left(1-\frac{1}{\eta}\right)})$. As for the case $\th<1$, H\"older inequality implies that $\E\left[\sup_{t\in[0,T]}||R_t^{\e}||^\th\right]=O(\e^{\frac{\th}{2}\left(1-\frac{1}{\eta}\right)})$.
\end{proof}

Thus we can show that the stochastic Hamiltonian system \eqref{slo} uniformly converges to the homogenized equation \eqref{limiting} in moment as follows.
\begin{theorem}
\label{result1}
{\bf (Convergence to the limiting equation in moment)}
Suppose Assumption 1-7 holds. Let $x_t^{\e}$ be the solution of SDE \eqref{slo} with initial condition $(p_0^{\e},q_0^{\e})$ and $q_t$ be the solution of SDE \eqref{limiting}with initial condition $q_0$. Also suppose that for every $\e>0, \eta>1$, the initial condition satisfies integrable conditions $\E[||q_0^{\e}||^\th]<\infty, \E[||q_0||^\th]<\infty$ and $\E[||q_0^{\e}-q_0||^\th]=O(\e^{\beta})$. Then for every $T>0$ and $\th<\eta$, we have
\begin{equation}
\E\left[\sup_{t\in[0,T]}||q_t^{\e}-q_t||^\th\right]=O(\e^{\beta}) \ as \ \e\to 0.
\end{equation}
\end{theorem}

\begin{proof}
First let $\th\ge2$. Define a vector $\widetilde F(t,x)$ and a matrix $\widetilde \sigma(t,x)$ as follows respectively
\begin{equation*}
\widetilde F_i(t,x)=(\gamma^{-1})_i^j(t,q)(\partial_{p_j}K(t,x)+\partial_{p_j}V(t,q)+F_j(t,x)),
\end{equation*}
\begin{equation*}
\widetilde \sigma_i^{\rho}(t,x)=(\gamma^{-1})_i^j(t,q)\sigma_j^{\rho}(t,x).
\end{equation*}

Hence we can rewrite Eq.\eqref{original} as
\begin{equation}
(q_t^{\e})_i=(q_0^{\e})_i+\int_0^t \widetilde F_i(s,x_s^{\e})ds+\int_0^t\widetilde \sigma_i^{\rho}(s,x_s^{\e})d(L_s)_{\rho}+ S_i(t,x_t^{\e})+(R_t^{\e})_i,
\end{equation}
and Eq.\eqref{limiting} as
\begin{equation}
(q_t)_i=(q_0)_i+\int_0^t\widetilde F_i(s,x_s)ds+\int_0^t\widetilde \sigma_i^{\rho}(s,x_s)d(L_s)_{\rho}+S_i(t,x_t).
\end{equation}
Therefore, we obtain the following estimation
\begin{equation}\label{T2E1}
\begin{aligned}
&\E\left[\sup_{s\in[0,t]}||q_s^{\e}-q_s||^\th\right]\\
&\le \tilde{C}\E\left[\sup_{s\in[0,t]}\left(||q_0^{\e}-q_0||^\th+\left|\left|\int_0^s \widetilde F_i(r,x_r^{\e})-\widetilde F_i(r,x_r)dr\right|\right|^\th+\left|\left|\int_0^s\widetilde \sigma_i^{\rho}(r,x_r^{\e})-\sigma_i^{\rho}(r,x_r)d(L_r)_{\rho}\right|\right|^\th\right.\right.\\
&+\left.\left.||S_i(s,x_s^{\e})-S_i(s,x_s)||^\th+||R_s^{\e}||^\th\right)\right].
\end{aligned}
\end{equation}

By the Lipschitz property of $\widetilde F$ and $\widetilde \sigma$ due to Assumptions, we have
\begin{equation}
\label{T1F}
\begin{aligned}
\E\left[\sup_{s\in[0,t]}\left|\left|\int_0^s \widetilde F_i(r,x_r^{\e})-\widetilde F_i(r,x_r)dr\right|\right|^\th\right]&\le \E\left[\sup_{s\in[0,t]}s^{\th-1}\int_0^s ||\widetilde F_i(r,x_s^{\e})-\widetilde F_i(r,x_s)||^\th ds\right]\\
&\le T^{\th-1}\E\left[\int_0^t\left|\left|F_i(r,x_r^{\e})-\widetilde F_i(r,x_r)\right|\right|^\th dr\right]\\
&\le \tilde{C}\left(\int_0^t\E[\sup_{r\in[0,s]}||q_r^{\e}-q_r||^\th]ds+\sup_{s\in[0,t]}\E[||p_s^{\e}||^\th]\right),
\end{aligned}
\end{equation}
and
\begin{equation}
\label{T1sigma}
\begin{aligned}
&\E\left[\sup_{s\in[0,t]}\left|\left|\int_0^s\widetilde \sigma_i^{\rho}(r,x_r^{\e})-\widetilde\sigma_i^{\rho}(r,x_r)d(L_r)_{\rho}\right|\right|^\th\right]\\
&\le \tilde{C}\E\left[\sup_{s\in[0,t]}\left(\left|\left|\int_0^s\int_{\mathbb{R}^d\backslash\{0\}}(\widetilde \sigma_i^{\rho}(r,x_r^{\e})-\widetilde\sigma_i^{\rho}(r,x_r))x\widetilde N(dr,dx)\right|\right|^\th+\left|\left|\int_0^s\int_{|x|>1}(\widetilde \sigma_i^{\rho}(r,x_r^{\e})-\widetilde\sigma_i^{\rho}(r,x_r))x\nu(dx)dr\right|\right|^\th\right)\right]\\
&\le \tilde{C}\left(\E\left[\left(\int_0^t\int_{\mathbb{R}^d\backslash\{0\}}||\widetilde \sigma_i^{\rho}(s,x_s^{\e})-\widetilde\sigma_i^{\rho}(s,x_s)||^2|x|^2\nu(dx)ds\right)^{\frac{\th}{2}}\right]\right.\\
&+ \left.\E\left[\int_0^t\int_{\mathbb{R}^d\backslash\{0\}}||\widetilde \sigma_i^{\rho}(s,x_s^{\e})-\widetilde\sigma_i^{\rho}(s,x_s)||^\th|x|^\th\nu(dx)ds\right]
+\E\left[\int_0^t\left|\left|\int_{|x|>1}x\nu(dx)(\widetilde \sigma_i^{\rho}(s,x_s^{\e})-\widetilde\sigma_i^{\rho}(s,x_s))\right|\right|^\th ds\right]\right)\\
&\le \tilde{C}\E\left(\int_0^t||\widetilde \sigma_i^{\rho}(s,x_s^{\e})-\widetilde\sigma_i^{\rho}(s,x_s)||^\th ds\right)\\
&\le \tilde{C}\left(\int_0^t\E[\sup_{r\in[0,s]}||q_r^{\e}-q_r||^\th]dr+\sup_{s\in[0,t]}E[||p_s^{\e}||^\th]\right).
\end{aligned}
\end{equation}
We can also get a similar bound for the noise-induced term
\begin{equation}
\label{T1S}
\E\left[\sup_{s\in[0,t]}||S_i(s,x_s^{\e})-S_i(s,x_s)||^\th\right]
\le \tilde{C}\left(\int_0^t\E[\sup_{r\in[0,s]}||q_r^{\e}-q_r||^\th]dr+\sup_{s\in[0,t]}\E[||p_s^{\e}||^\th]\right)\\.
\end{equation}
Consequently, estimations \eqref{T1F}-\eqref{T1S} together with Proposition \ref{proposition1} and Lemma \ref{R} yield that
\begin{equation}
\begin{aligned}
\E\left[\sup_{s\in[0,t]}||q_s^{\e}-q_s||^\th\right]\le \tilde{C}\int_0^t\E\left[\sup_{r\in[0,s]}||q_r^{\e}-q_r||^\th\right]ds+O(\e^\beta),
\end{aligned}
\end{equation}
for all $t\in[0,T]$. If $\E\left[\sup_{s\in[0,t]}||q_s^{\e}-q_s||^\th\right]\in L^1[0,T]$. Then Gronwall's inequality implies
\begin{equation}
\E\left[\sup_{s\in[0,t]}||q_s^{\e}-q_s||^\th\right]\le O(\e^\beta)e^{\tilde{C}t},
\end{equation}
which is precisely the result we want to prove.
Indeed,
\begin{equation*}
\begin{aligned}
\E\left[\sup_{t\in[0,T]}||q_t^\e||^\th\right]&\le C\left(\E\left[\sup_{t\in[0,T]}||q_0^\e||^\th\right]+\E\left[\sup_{t\in[0,T]}\left|\left|\int_0^t \widetilde F(s,x_s^\e)ds\right|\right|^\th\right]\right.\\
&\left.+\E\left[\sup_{t\in[0,T]}\left|\left|\int_0^t\widetilde \sigma^{\rho}(s,x_s^\e)d(L_s)_{\rho}\right|\right|^\th\right]+\E\left[\sup_{t\in[0,T]}|| S(t,x_t^\e)||^\th\right]+\E\left[\sup_{t\in[0,T]}||(R_t^\e)||^\th\right]\right)\\
&< \infty,
\end{aligned}
\end{equation*}
and similarly we can get $\E\left[\sup_{t\in[0,T]}||q_t||^\th\right]<\infty$.

\end{proof}

\subsection{\textbf{Extension}}
In this section, we relax some assumptions that we make before. Actually we can extend all Lipschitz conditions to locally Lipschitz condition and remove all boundedness conditions. Organize and summarize the assumptions in the previous article, now we give a complete theorem.
\begin{theorem}
\label{result2}
{\bf (Convergence to the limit equation in probability)}
Suppose the family of Hamiltonians have the form
\begin{equation*}
H^{\e}(t,q,p)=K^{\e}(t,q,p)+V(t,q)=K(\e,t,q,p/\sqrt{\e})+V(t,q),
\end{equation*}
and the following conditions hold:\\
1. The function $K^{\e}(t,q,p)$  is non-negative and $\mathcal{C}^2$.\\
2. There exist constant $C>0, M_1>0$ such that
\begin{equation*}
\max{\{|\partial_tK(\e,t,q,z)|,||\nabla_qK(\e,t,q,z)||,||\nabla_zK(\e,t,q,z)||\}}\le M_1+CK(\e,t,q,z).
\end{equation*}
3. There exist constant $c>0, M_2\ge0$ such that
\begin{equation*}
||\nabla_zK(\e,t,q,z)||^2+M_2\ge cK(\e,t,q,z).
\end{equation*}
4. For every $T>0$, there exist constant $c>0, \eta>1$ such that
\begin{equation*}
K(\e,t,q,z)\ge c||z||^{\eta}.
\end{equation*}
5. The potential energy function $V(t,q)$ is $\mathcal{C}^1$.\\
6. The dissipative coefficient $\gamma$ is $\mathcal{C}^2$, independent of $p$ and symmetric with eigenvalues bounded below by a constant $\lambda>0$.\\
7. The external force $F$ and noise intensity coefficient $\sigma$ are continuous and locally Lipschitz.\\
Let $x_t^{\e}$ be the solution of SDE \eqref{slo} with initial condition $(p_0^{\e},q_0^{\e})$ and $q_t$ be the solution of SDE \eqref{limiting}with initial condition $q_0$. Also suppose that for every $\e>0$ and $\th\in(0,\eta)$, the initial condition satisfies integrable conditions $\E[||q_0^{\e}||^\th]<\infty, \E[||q_0||^\th]<\infty$ and $\E[||q_0^{\e}-q_0||^\th]=O(\e^\beta)$.
Then  for every $T>0,\delta>0$ we have
\begin{equation}
\lim_{\e\to 0} \P\left(\sup_{t\in[0,T]}||q_t^{\e}-q_t||>\delta\right)=0.
\end{equation}
\end{theorem}

\begin{proof}
Let $\chi: \mathbb{R}^n\to [0,1]$ be a $C^{\infty}$ function.
Define
\begin{equation*}
\begin{aligned}
V_r(t,q)=\chi_r(q)V(t,q), F_r(t,x)=\chi_r(q)\chi_r(p)F(t,x), \sigma_r(t,x)=\chi_r(q)\chi_r(p)\sigma(t,x),\\
K(\e,t,q,z)=\chi_r(z)K(\e,t,q,z), \gamma_{r}(t,q)=\chi_r(q)\gamma(t,q)+(1-\chi_r(q))\lambda I
\end{aligned}
\end{equation*}
Replacing the function $V,F,K,\gamma,\sigma$ in \eqref{slo} by $V_r, F_r, K_r, \gamma_r, \sigma_r$, we arrive at an SDE satisfying the condition in Theorem \ref{result1}. Let $x_t^{r,\e}$ be solution to the corresponding SDE. Similarly, let $q_t^r$ be the solution to the corresponding limiting SDE \eqref{limiting}. Proposition \ref{proposition1} and Theorem \ref{result1} imply that, for every $T>0, \ \eta>1$ and $\th\in(0,\eta)$
\begin{equation}\label{extension1}
\E\left[\sup_{t\in[0,T]}||p_t^{r,\e}||^\th\right]=O(\e^{\frac{\th}{2}-\frac{\th}{2\eta}}) \ \text{as} \ \e\to 0,
\end{equation}
and
\begin{equation}\label{extension2}
\E\left[\sup_{t\in[0,T]}||q_t^{r,\e}-q_t^r||^\th\right]=O(\e^\beta) \ \text{as} \ \e\to 0.
\end{equation}
We will use this result to prove that $q_t^{\e}$ converges to $q_t$ in probability.

Denfine stopping times $\tau_r^{\e}=\inf\{t:||q_t^{\e}||\ge r\}$, $\eta_r^{\e}=\inf\{t:||p_t^{\e}||\ge \e r\}$ and $\tau_r=\inf\{t:||q_t||\ge r\}$. The drifts and diffusions of the modified and unmodified SDEs agree on the ball $\{||q||<r, ||p||<\e r\}$. Hence
\begin{equation*}
q_{\tau_r^{\e}\land\eta_r^{\e}\land t}^{\e}=q_{\tau_r^{\e}\land\eta_r^{\e}\land t}^{r,\e}, \ q_{\tau_r\land t}=q_{\tau_r\land t}^r \ \text{for all} \ t\ge0 \ \text{a.s.}
\end{equation*}
For every $T>0, \delta>0$, we deduce that
\begin{equation}\label{E1}
\begin{aligned}
&\P\left(\sup_{t\in[0,T]}||q_t^{\e}-q_t||>\delta\right)\\
=&\P\left(\tau_r\land\tau_r^{\e}\land\eta_r^{\e}>T, \sup_{t\in[0,T]}||q^{\e}_{\tau_r^{\e}\land\eta_r^{\e}\land t}-q_{\tau_r\land t}||>\delta\right)+\P\left(\tau_r\land\tau_r^{\e}\land\eta_r^{\e}\le T, \sup_{t\in[0,T]}||q_t^{\e}-q_t||>\delta\right)\\
=&\P\left(\tau_r\land\tau_r^{\e}\land\eta_r^{\e}>T, \sup_{t\in[0,T]}||q^{r,\e}_t-q_t^r||>\delta\right)+\P\left(\tau_r\land\tau_r^{\e}\land\eta_r^{\e}\le T, \sup_{t\in[0,T]}||q_t^{\e}-q_t||>\delta\right)\\
\le&\P\left(\sup_{t\in[0,T]}||q^{r,\e}_t-q_t^r||>\delta\right)+\P\left(\tau_r\land\tau_r^{\e}\land\eta_r^{\e}\le T\right),
\end{aligned}
\end{equation}
where the first term on the right hand side converges to 0 as $\e\to 0$ by \eqref{extension2}. Then we focus on the second term,
\begin{equation}\label{E2}
\begin{aligned}
&\P\left(\tau_r\land\tau_r^{\e}\land\eta_r^{\e}\le T\right)\\
=&\P(\tau_r\le T)+\P\left(\tau_r>T, \tau_r^{\e}\land \eta_r^{\e}\le T\right)\\
\le&\P(\tau_r\le T)+\P\left(\sup_{t\in[0,T]}||q_t^{r,\e}-q_t^r||>1\right)+\P\left(\tau_r>T, \tau_r^{\e}\land\eta_r^{\e}\le T,\sup_{t\in[0,T]}||q_t^{r,\e}-q_t^r||\le 1\right)\\
\le&\P\left(\sup_{t\in[0,T]}||q_t^r||>r\right)+\P\left(\sup_{t\in[0,T]}||q_t^{r,\e}-q_t^r||>1\right)+ \P\left(\tau_r^{\e}\land\eta_r^{\e}\le T, ||q_{\tau_r^{\e}\land\eta_r^{\e}\land T}^{r,\e}-q^r_{\tau_r^{\e}\land\eta_r^{\e}\land T}||\le 1\right)\\
\le&\P\left(\sup_{t\in[0,T]}||q_t^r||>r\right)+\P\left(\sup_{t\in[0,T]}||q_t^{r,\e}-q_t^r||>1\right)+ \P\left(\eta_r^{\e}>T, \tau_r^{\e}\le T, ||q_{\tau_r^{\e}\land T}^{r,\e}-q^r_{\tau_r^{\e}\land T}||\le 1\right)\\
&+\P\left(\eta_r^{\e}\le T, ||q_{\tau_r^{\e}\land\eta_r^{\e}\land T}^{r,\e}-q_{\tau_r^{\e}\land\eta_r^{\e}\land T}||\le 1\right).
\end{aligned}
\end{equation}
Note that when $\tau_r^{\e}\le T$, we have  $||q_{\tau_r^{\e}\land T}||\ge r$. Hence by $||q_{\tau_r^{\e}\land T}^{r,\e}-q^r_{\tau_r^{\e}\land T}||\le 1$, we can deduce
\begin{equation*}
||q^r_{\tau_r^{\e}\land T}||\ge ||q_{\tau_r^{\e}\land T}^{r,\e}||-||q_{\tau_r^{\e}\land T}^{r,\e}-q^r_{\tau_r^{\e}\land T}||> r-1.
\end{equation*}
This implies that
\begin{equation}\label{E3}
\P\left(\tau_r^{\e}\le T, ||q_{\tau_r^{\e}\land T}^{r,\e}-q^r_{\tau_r^{\e}\land T}||\le 1\right)\le \P\left(||q^r_{\tau_r^{\e}\land T}||>r-1\right)\le \P\left(\sup_{t\in[0,T]}||q^r_t||>r-1\right).
\end{equation}
Combining \eqref{E1},\eqref{E2} and \eqref{E3} together, we have
\begin{equation}
\begin{aligned}
&\P\left(\sup_{t\in[0,T]}||q_t^{\e}-q_t||>\delta\right)\\
&\le \P\left(\sup_{t\in[0,T]}||q^{r,\e}_t-q_t^r||>\delta\right)+\P\left(\sup_{t\in[0,T]}||q_t^r||>r\right)+\P\left(\sup_{t\in[0,T]}||q_t^{r,\e}-q_t^r||>1\right)\\
&+\P\left(\sup_{t\in[0,T]}||q^r_t||>r-1\right)+\P\left(\eta_r^{\e}\le T\right).
\end{aligned}
\end{equation}
On the other hand, by Chebyshev inequality and \eqref{extension1}, we have
\begin{equation}
\P\left(\eta_r^{\e}\le T\right)\le \P\left(\sup_{t\in[0,T]}||p_t^{r,\e}||>\e r\right)\le (\e r)^{-2}\E\left[\sup_{t\in[0,T]}||p_t^{r,\e}||^2\right]=O(\e^{-1-\frac{1}{\eta}})r^{-2}.
\end{equation}
Then if we let $r^{-1}=o(\e^{\frac{1}{2}\left(1+\frac{1}{\eta}\right)})$, i.e., the speed of $r$ goes to infinity faster than $\e^{-\frac{1}{2}\left(1+\frac{1}{\eta}\right)}$.  We have
\begin{equation}
\P\left(\sup_{t\in[0,T]}||q_t^{\e}-q_t||>\delta\right)\to 0 \ \text{as} \ r\to\infty, \ \e\to 0
\end{equation}
by the non-explosion property of $q_t^r$.
\end{proof}

\section{An Example}
In this section, we present a prototypical example with Hamiltonian $H(m,t,q,p)=\frac{p^2}{2m}+V(t,q)$, where $m$ is the mass of a particle. In this case, the small mass limit is also called Smoluchowski-Kramers limit. We consider the stochastic Hamiltonian system with external force $F(t,x)$ and L\'evy noise $L_t$
\begin{equation}
\begin{aligned}
&dq_t^m=\frac{1}{m}p_t^mdt,\\
&dp_t^m=\left(\frac{1}{m}\gamma(t,q_t^m)p_t^m-\nabla_qV(t,q_t^m)+F(t,x_t^m)\right)dt+\sigma(t,x_t^m)dL_t.
\end{aligned}
\end{equation}
By Proposition \ref{proposition1}, $p_t^m$ converges to zero. Then the homogenized equation in the small mass limit is
\begin{equation}
dq_t=\gamma^{-1}(t,q_t)(\nabla_qV(t,q_t)+F(t,q_t,0))dt+\gamma^{-1}(t,q_t)\sigma(t,q_t,0)dL_t+S(t,q_t),
\end{equation}
where the noise induced drift is
\begin{equation}\label{noisedriftex}
S_i(t,q_t)=\int_0^t \int_{\mathbb{R}^d\backslash\{0\}} \partial_{q^h}(\gamma^{-1})_i^j(t,q_t)\int_0^{\infty}\left(e^{-y\gamma(s,q_s)}\right)_j^a\left(e^{-y\gamma(s,q_s)}\right)_l^bdy \sigma_a^k(s,q_t,0)\sigma_b^l(s,q_t,0)z_kz_lN(ds,dz).
\end{equation}

Moreover, when dissipative coefficient $\gamma$ is independent of $q$, the noise-induced drift \eqref{noisedriftex} vanish, and the homogenized equation becomes
\begin{equation}
dq_t=\gamma^{-1}(t)(\nabla_qV(t,q_t)+F(t,q_t,0))dt+\gamma^{-1}(t)\sigma(t,q_t,0)dL_t.
\end{equation}
This result coincide with that in \cite{zhang2008smoluchowski}.

\section{Conclusion and Discussion}
In this paper, we derive the small mass limiting equation for a class of Hamiltonian systems with multiplicative L\'evy noise. Some interesting results appear. If the Hamiltonian function $H(\e,q,p)$ possesses appropriate properties, then momentum $p$ will always converge to zero in finite time under uniform norm. The noise-induced drift term induced by pure jump L\'evy noise is a Poisson process, which is rather different from that induced by Gaussian noise \cite{birrell2018homogenization}. Our results could be applied to a class of stochastic Hamiltonian systems, such as a small mass particle in force field with state-dependent friction and a particle on a Riemannian manifold.

However, we have to mention that the pure jump L\'evy noises in this paper have finite moment. In other words, it has bounded jumps. Large jumps could lead to some unpredictable dynamics although interlacing techniques allow us to deal with it. Hence an interesting problem is that how to accurately deal with L\'evy noise without finite moments such as $\alpha$-stable L\'evy noise, which will be studied in the future.

\section*{Acknowledgments}
The authors would like to thank Lingyu Feng, Jianyu Hu, Pingyuan Wei, Shenglan Yuan and Yanjie Zhang for helpful discussions. This work was partly supported by NSFC grants 11771449 and 11531006.

\section*{Appendix}
\begin{appendices}

\section{Non-explosion of solution}
\renewcommand\thesection{\Alph{section}}
In Appendix, we will prove that the solution of SDE \eqref{slo} and limit equation are existence and unique under Assumption 1-4.
\begin{lemma}
Under Assumption 1-4, there exists a unique non-explosive solution to \eqref{slo} in finite time interval $[0,T]$.
\end{lemma}
\begin{proof}
First, we can verify that SDE with Assumption 1-3 satisfies Lipschitz condition and one side growth condition (refer to \cite{applebaum2009levy}) in every bounded cylinder $I\times U(R)$, where $U(R)$ is a ball with radius $R$. Then, we will prove that there is no explosion. Let $\tau_n$ be the first exit time of $x_t^{\e}$ from the ball $B(0,n)$.
From the right-continuity of the process $x^{\e}_t$ we infer that
\begin{equation}
|x^{\e}_{\tau_n}|\ge n.
\end{equation}
Define a function $U^{\e}(t,x_t^{\e})=||q_t^{\e}||^{2\eta}+K^{\e}(t,x_t^{\e})$.
By Assumption 4, we obtain that
\begin{equation}
\begin{aligned}
U^{\e}(\tau_n,x^{\e}_{\tau_n})&=||q_{\tau_n}^{\e}||^{2\eta}+K^{\e}(\tau_n,x^{\e}_{\tau_n})\\
&\ge ||q_{\tau_n}^{\e}||^{2\eta}+c\e^{-\eta}||p_{\tau_n}^{\e}||^{2\eta}\\
&\ge \min\{1,c\e^{-\eta}\}||x_{\tau_n}^{\e}||^{2\eta}\\
&\ge c|n|^{2\eta}.
\end{aligned}
\end{equation}
On the other hand, we have
\begin{equation}
\begin{aligned}
&\E\left[U^{\e}(t\land\tau_n\land T,x^{\e}_{t\land\tau_n\land T})\right]\\
&=\E\left[U^{\e}(t\land\tau_n\land T,x^{\e}_{t\land\tau_n\land T})1_{\{\tau_n\land T\ge t\}}\right]+\E\left[U^{\e}(t\land\tau_n\land T,x^{\e}_{t\land\tau_n\land T})1_{\{\tau_n\land T< t\}}\right]\\
&=\E\left[U^{\e}(t,x^{\e}_t)1_{\{\tau_n\land T\ge t\}}\right]+\E\left[U^{\e}(\tau_n\land T,x^{\e}_{\tau_n\land T})1_{\{\tau_n\land T< t\}}\right]\\
&=\E\left[U^{\e}(t,x^{\e}_t)1_{\{\tau_n\land T\ge t\}}\right]+\E\left[U^{\e}(\tau_n,x^{\e}_{\tau_n})1_{\{\tau_n<T\}}1_{\{\tau_n< t\}}\right]+\E\left[U^{\e}(T,x^{\e}_{T})1_{\{\tau_n\ge T\}}1_{\{T< t\}}\right]\\
&\ge \E\left[U^{\e}(\tau_n,x^{\e}_{\tau_n})1_{\{\tau_n< t\}}\right].
\end{aligned}
\end{equation}
Therefore, for all $n\in\mathbb{N}$
\begin{equation}
\P(\tau_n<t)\le c^{-1}n^{-2\eta}\E\left[U^{\e}(t\land\tau_n\land T,x^{\e}_{t\land\tau_n\land T})\right].
\end{equation}
Notice that by Theorem \ref{Esup} we have
\begin{equation}
\E\left[U^{\e}(t\land\tau_n\land T,x^{\e}_{t\land\tau_n\land T})\right]\le \E\left[\sup_{t\in[0,T]}||q_t^{\e}||^{2\eta}\right]+\E\left[\sup_{t\in[0,T]}K^{\e}(t,x_t^{\e})\right]=O(1).
\end{equation}
Hence,
\begin{equation}
\lim_{n\to\infty}\P(\tau_n<t)=0 \ \text{for all} \ t.
\end{equation}
That is the desired assertion, as required.
\end{proof}

\section{Proofs of \eqref{J4} and \eqref{J7}}
We give calculations for estimations of \eqref{J4} and \eqref{J7} in remainder term.\\
{\bf{Proof of \eqref{J4}}.}
By Assumption 7 we can deduce that the function $\partial_q(\gamma^{-1})(t,q)G(t,q)$ is bounded and $\mathcal{C}^1$. Let $f(t,q)=\partial_q(\gamma^{-1})(t,q)G(t,q)$. We have
\begin{equation}
\begin{aligned}
J_4&=\E\left[\sup_{t\in[0,T]}\left|\left|\int_0^t\partial_{q^h}(\gamma^{-1})_i^j(s,q_s^{\e})G_{jh}^{ab}(s,q_s^{\e})d((p_s^{\e})_a(p_s^{\e})_b)\right|\right|^\th\right]\\
&\le \E\left[\sup_{t\in[0,T]}\left|\int_0^t f(s,q_s^\e)d((p_s^{\e})_i(p_s^{\e})_j)\right|^\th\right].
\end{aligned}
\end{equation}
Since $f(s,q_s^{\e})$ is a $C^1$-semimartingale, using integration by parts formula we obtain
\begin{equation}
\begin{aligned}
\int_0^t f(s,q_s^\e)d((p_s^{\e})_i(p_s^{\e})_j)&=f(t,q_t^\e)(p_t^{\e})_i(p_t^{\e})_j -f(0,q_0^\e)(p_0^{\e})_i(p_0^{\e})_j\\
&-\int_0^t (p_s^{\e})_i(p_s^{\e})_j\left(\partial_sf(s,q_s^{\e})+\nabla_qf(s,q_s^{\e})\nabla_pH^{\e}(s,x_s^{\e})\right)ds.
\end{aligned}
\end{equation}
Hence, for $\th\ge1$, we have
\begin{equation}
\begin{aligned}
J_4&\le 3^{\th-1}\left(2||f||_{\infty}^\th\E[\sup_{t\in[0,T]}||p_t^{\e}||^{2\th}]+ \E\left[\sup_{t\in[0,T]}\left|\int_0^t||p_s^{\e}||^2\left(||\partial_s f||_{\infty}+||\nabla_q f||_{\infty}|\nabla_pK^\e(s,x_s^\e)|\right)ds\right|^\th\right]\right)\\
&\le 3^{\th-1}\left(2||f||_{\infty}^\th\E[\sup_{t\in[0,T]}||p_t^{\e}||^{2\th}]+ \E\left[\sup_{t\in[0,T]}\left|\int_0^t||p_s^{\e}||^2\left(||\partial_s f||_{\infty}+||\nabla_q f||_{\infty}\frac{1}{\sqrt{\e}}(M_1+CK^{\e}(s,x_s^{\e}))\right)ds\right|^\th\right]\right)\\
&\le  3^{\th-1}2||f||_{\infty}^\th\E[\sup_{t\in[0,T]}||p_t^{\e}||^{2\th}]+6^{\th-1}T^{\th-1}\E \left[\int_0^T||p_s^\e||^{2\th}\left(||\partial_s f||_{\infty}^\th+M_1^\th||\nabla_q f||_{\infty}^\th\e^{-\frac{\th}{2}}+C^\th K^\e(s,x_s^\e)^\th\e^{-\frac{\th}{2}}\right)ds\right].
\end{aligned}
\end{equation}
\qed

{\bf{Proof of \eqref{J7}}.}
Applying Kunita's first inequality \cite{applebaum2009levy} on $J_7$, we have
\begin{equation}
\begin{aligned}
J_7&=2^{\th-1}\E\left[\sup_{t\in[0,T]}\left|\left|\int_0^t\int_{\mathbb{R}^d\backslash \{0\}}\partial_{q^h}(\gamma^{-1})_i^j(s,q_s^{\e})G_{jh}^{ab}(s,q_s^{\e})(p_{s-}^{\e})_a\sigma_b^{\rho}(s,x_s^{\e})x\widetilde N(ds,dx)\right|\right|^\th\right.\\
&+\left.\sup_{t\in[0,T]}\left|\left|\int_0^t\int_{|x|>1}\partial_{q^h}(\gamma^{-1})_i^j(s,q_s^{\e})G_{jh}^{ab}(s,q_s^{\e}) (p_{s-}^{\e})_a\sigma_b^{\rho}(s,x_s^{\e})x\nu(dx)ds\right|\right|^\th\right]\\
&\le 2^{\th-1}D(\th)\E\left[\left(\int_0^T\int_{\mathbb{R}^d\backslash \{0\}}\left|\left|\partial_{q^h}(\gamma^{-1})_i^j(s,q_s^{\e})G_{jh}^{ab}(s,q_s^{\e}) (p_{s-}^{\e})_a\sigma_b^{\rho}(s,x_s^{\e})x\right|\right|^2\nu(dx)ds\right)^{\frac{\th}{2}}\right]\\
&+2^{\th-1}\E\left[\int_0^T\int_{\mathbb{R}^d\backslash \{0\}}\left|\left|\partial_{q^h}(\gamma^{-1})_i^j(s,q_s^{\e})G_{jh}^{ab}(s,q_s^{\e}) (p_{s-}^{\e})_a\sigma_b^{\rho}(s,x_s^{\e})x\right|\right|^\th\nu(dx)ds\right]\\
&+2^{\th-1}T^{\th}C\left(\int_{|x|>1}|x|\nu(dx)\right)^\th\E\left[\sup_{t\in[0,T]}||p_t^{\e}||^\th\right]\\
&\le 2^{\th-1}\left(D(\th)T^{\frac{\th}{2}}C\int_{\mathbb{R}\backslash\{0\}}|x|^2\nu(dx)^{\frac{\th}{2}}+ TC\int_{\mathbb{R}\backslash\{0\}}|x|^\th\nu(dx)+T^\th C\left(\int_{|x|>1}|x|\nu(dx)\right)^\th\right)\sup_{t\in[0,T]}\E\left[||p_t^{\e}||^\th\right].
\end{aligned}
\end{equation}
We have to mention that Kunita's first inequality holds for $\th\ge2$. Actually $J_7\le\tilde{C}\sup_{t\in[0,T]}\E\left[||p_t^{\e}||^\th\right]$ still holds for $\th\in[1,2)$ since $\sup_{t\in[0,T]}\E\left[||p_t^{\e}||^\th\right]=O(\e^{\frac{\th}{2}})$ for $\th\in(0,2\eta)$.
\qed
\end{appendices}

\section*{References}

\bibliographystyle{elsarticle-num}
\bibliography{homogenization}

\end{document}